\documentclass[11pt]{amsart}

\advance\hoffset by -0.5 truecm
\usepackage[pagewise]{lineno}
\usepackage{amssymb}
\usepackage{amsmath}
\usepackage{graphicx}
\usepackage{color}

\newtheorem{Theorem}{Theorem}[section]
\newtheorem{Lemma}[Theorem]{Lemma}
\newtheorem{Corollary}[Theorem]{Corollary}

\newtheorem{Proposition}[Theorem]{Proposition}

\newtheorem{Remark}[Theorem]{Remark}

\def \dim{{\mbox {dim}}\,}

\def\N{{\mathbb N}}
\def\R{{\mathbb R}}

\begin{document}
	\title[Hardy-Sobolev equations on Riemannian manifolds] {On solutions to Hardy-Sobolev equations on Riemannian manifolds}

	\author[G. Henry]{Guillermo Henry}\thanks{G. Henry is partially supported by PICT-2020-01302  grant  from ANPCyT and UBACYT 2023-2025.}
	\address{Departamento de Matem\'atica, FCEyN, Universidad de BuenosAires and IMAS,   CONICET-UBA, Ciudad Universitaria, Pab. I., C1428EHA, Buenos Aires, Argentina and CONICET, Argentina.}
	\email{ghenry@dm.uba.ar}

	\author[J. Petean]{Jimmy Petean}\thanks{J. Petean is supported
		by SECIHTE proyecto CBF-2025-I-1449}
	\address{CIMAT, A.P. 402, 36000, Guanajuato. Gto., M\'exico}
	\email{jimmy@cimat.mx}


	\date{}

	\begin{abstract} Let $(M,g)$ be a closed  Riemannian manifold of dimension at least $3$. Let $S$ be the union of 
the  focal submanifolds of an isoparametric function on $(M,g)$.   
In this article we address the existence of solutions of the Hardy-Sobolev type equation
		$\Delta_g u+K(x)u=\frac{u^{q-1}}{\left(d_{S}(x)\right)^s}$, where $d_{S}(x)$ is the distance from $x$ to $S$ 
		and $q>2$. In particular, we will prove the existence of infinite sign-changing solutions to the
equation.
		
	\end{abstract}
	
	\maketitle

	\section{Introduction}

On a closed Riemannian manifold of dimension $n\geq 3$ we will study a Hardy-Sobolev equation of the form

\begin{equation}\label{H-SequationSubmanifold}
\Delta_g u(x)+K(x)u(x)=\frac{u^{q-1}(x)}{\left(d_g(S,x)\right)^s},
\end{equation}

\noindent
where  $s\in[0,2)$, $q>2$, $S$ is a submanifold of $M$, $\Delta_g =-div_g  \ ( \nabla )$ is the (non-negative) Laplace-Beltrami operator and $d_g (S, -) :M \rightarrow \R$ denotes the distance to $S$. 	
The function $K$ will be assumed to be continuous. 

\medskip

Hardy-Sobolev equations have been extensively studied over the years.  The literature on these equations in the Euclidean 
setting is vast.  Positive solutions to the equations are related to the best constant in the
Hardy-Sobolev inequalities. N. Ghoussoub and C. Yuan considered in \cite{G-Y} such  equations 
on a bounded open set $\Omega$ of $\R^n$ when the singularity is a point in the interior of $\Omega$.
The case when the singularity is in the boundary of the domain is studied  by N. Ghoussoub and F. Robert
in \cite{Ghoussoub-Robert}. Other results about the existence of positive and sign changing solutions
on bounded domains of Euclidean space were obtained for instance in \cite{Ghoussoub, Kang, Kang-Peng}.
In these articles the reader can find references on other results on Hardy-Sobolev equations on open subsets
of Euclidean space. 
 
\medskip

Given the Riemannian manifold $(M,g)$, we denote by  $d_g$ the distance function induced by $g$,
$d_g : M\times M \rightarrow \R_{\geq 0}$.  H.  Jaber considered in \cite{Jaber} the  {\it Hardy-Sobolev type equation},  
	
\begin{equation}\label{H-SequationManifolds}
\Delta_g u(x)+K(x)u(x)= \frac{u^{2^*(s)-1}(x)}{\big(d_g(x_0,x)\big)^s},
\end{equation}
	
\noindent
on $M$, 
where $K \in C^{0}(M)$,   $x_0\in M$, $s\in [0,2)$ and $2^{*}(s)=\frac{2(n-s)}{n-2}$. 
The author  proved in \cite{Jaber} 
the existence of a positive solution $u\in C^0 (M) \cap H_1^2 (M)$, provided  $K(x_0)<c(n,s) Scal_g(x_0)$ (where $c(n,s)$ is a positive constant) if $n\geq 4$ or the mass of $\Delta_g+K$ is positive at $x_0$ if $n=3$, and assuming
that the  operator $\Delta_g +K$ is coercive. A similar problem was considered by
E. H. A.  Thiam in \cite{Thiam0}. 
The equation is critical, in the sense that $H_1^2 (M)$ is 
continuously embedded in  the weighted space 
$L^p  (M,  d_g ( -  ,  x_0  )^{-s}   )$ 
if $p\leq 2^* (s)$ and the embedding  is compact if the inequality is strict. It appears naturally 
in the study of the Hardy-Sobolev inequality:

\begin{equation}
\left(   \int_M    \frac{ | u |^{2^* (s)}}{ d_g (x,x_0 )^s}    dv_g       \right)^{\frac{2}{2^* (s)}}
   \leq A \int_M  \| \nabla u \|^2 dv_g    +B \int_M u^2 dv_g
\end{equation}

The best constant in this inequality, the minimum of the values of $A$ such that there exists $B$ for which the
inequality holds for every $u$,  was obtained by  H.  Jaber in \cite{Jaber3}. 
In \cite{Cheikh} H. Cheikh Ali studies the second best constant for the Hardy-Sobolev inequality on a closed 
Riemannian manifold.

It is also important to consider the more general equation

\begin{equation}\label{H-Sequation-q}
\Delta_g u(x)+K(x)u(x)= \frac{u^{q-1}(x)}{\big(d_g(x_0,x)\big)^s},
\end{equation}

\noindent
for any $q > 2$.

In \cite{Chen} W. Chen studies the existence of blow up solutions, considering $q$ close to the
critical exponent $2^* (s)$.

It is natural to consider the generalization of Equation \eqref{H-Sequation-q} 
when the singularity occurs on a positive dimensional embedded  submanifold 
$S$ of $M$:
	
\begin{equation}\label{H-SequationSubmanifold}
\Delta_g u(x)+K(x)u(x)=\frac{u^{q-1}(x)}{\left(d_g(S,x)\right)^s},
\end{equation}

\noindent
where $d_g (S, -) :M \rightarrow \R$ denotes the distance to $S$.

In \cite{Thiam}, E. H. A. Thiam studies positive solutions of Equation (\ref{H-SequationSubmanifold}) on a open subset of Euclidean space when
$S$ is a submanifold. Positive solutions when $S$ is  a curve are also studied by M. M. Fall and E. H. A. Thiam 
in \cite{Fall}.

H. Mesmar in \cite{Mesmar},  studied Equation \eqref{H-SequationSubmanifold} (including a perturbation term) in the case when $S$ is a minimal orbit of positive dimension 
of an action of a subgroup of isometries of $(M,g)$.
	
\medskip
	
In this article we are going to address the existence of solutions to Equation \eqref{H-SequationSubmanifold}  when $S$ is the union of focal submanifolds of an isoparametric function of $(M,g)$. Isoparametric functions on general Riemannian manifolds were introduced by Q. M. Wang in
\cite{Wang} , following the classical study in the case of space forms. 

A non-constant smooth function $f:M\longrightarrow [t_0,t_1]$ is called {\it isoparametric} if there exist $a$ and $b$  smooth functions such that  
	\begin{equation}\label{Cgradient}\| \nabla f \|^2 = b \circ f
	\end{equation}
	and
	\begin{equation}\label{Claplacian}
	\Delta_{g} f = a\circ f.
	\end{equation}

Equation \eqref{Cgradient} says that regular level sets of $f$ are parallel 
(constant distance) hypersurfaces. On the other hand, Equation \eqref{Claplacian}  and  \eqref{Cgradient} imply that the level sets of $f$ have constant mean curvature. 

Q. M. Wang proves in \cite{Wang} that the only critical values of an isoparametric function on a closed Riemannian manifold are the global maximum and minimum. 
The extremal level sets, $M_i=f^{-1}(t_i)$ ($i=0,1$), are embedded submanifolds, and they are known as  the {\it focal submanifolds} of $f$. We denote by $d_i=\dim(M_i)$ and by  
$k(f)=\min\{ d_0 , d_1  \}$. We let $S=M_0 \cup M_1$.
	
 We say that an isoparametric function is {\it proper} if the codimension of each focal submanifolds is greater than 1. J. Ge and Z. Tang proved in \cite{Ge-Tang1} that all the level sets of a proper isoparametric function are connected. 
    
 The family of closed Riemannian manifolds with isoparametrics functions is very rich. The case of the  sphere $S^n\subset \R^{n+1}$ with the metric of constant curvature is
the most studied case, see for instance the survey article by Q. -S. Chi \cite{Chi}.  In a more general situation, C. Qian and Z. Tang 
proved in \cite{Qian-Tang} that  if $f$ is any Morse-Bott function on a manifold $M$, whose critical level sets are two connected submanifolds of codimension greater than one, then there exists a Riemannian metric $g$ on $M$ for which $f$  is a proper isoparametric function.

Let $f$ be a proper isoparametric function of $(M,g)$, and let $D=d_g(M_0,M_1)$.  We say that  $u$ is an $f-$invariant function,  if it is constant along the level sets of $f$. This is equivalent to saying that there exists a real valued function $\varphi$ such that 
    \[u(x)=\varphi \circ {\bf{d}}(x),\]
	where ${\bf d}:M\longrightarrow [0,D]$ denote the distance to $M_0$.
	
	By $S_f$ and  $H^2_{1,f}(M)$ we denote the set of $f-$invariant functions and the set of $f-$invariant functions that belong to the Sobolev space $H^2_1(M)$, respectively.

\medskip

We will study solutions of the equation 

\begin{equation}\label{q} 
		\Delta_g u(x)+K(x)u(x)=\frac{u^{q-1}(x)}{\left(d(x)\right)^s},
		\end{equation} 
\noindent 
where $s\in (0,2)$, $q>2$, $d(x)=d_g (S,x)$ and $K$ is an $f$-invariant continuous function. We let
$K(x) = {\bf k} ({\bf d}  (x))$, ${\bf k} : [0 , D ] \rightarrow \R$. 
Isoparametric functions appear naturally in the study of certain partial differential equations on
Riemannian manifolds since they give a reduction of the equation to an ordinary differential equation.
In Section 2 we will show that an invariant function $u=\varphi \circ {\bf d}$ solves Equation (\ref{q}) if and only if
$\varphi $ solves the ordinary differential equation: 

\begin{equation}\label{O}
\varphi '' (t) - m(t) \varphi ' (t) -{\bf k} (t) \varphi (t) +  \frac{\varphi^{q-1} (t)}{(d(t))^s} =0 ,
\end{equation}

\noindent
where $m(t)$ denotes the mean curvature of ${\bf d}^{-1}(t)$ and $d : [0,D] \rightarrow \R$, $d(t)$ denotes the distance of ${\bf d}^{-1} (t)$  to 
$S= M_0 \cup M_1$: $d(t)=t$ if $t\in [0,D/2]$ and $d(t)=D-t$ if $t\in [D/2,D]$.

Solutions
of this equation are not  $C^2$ at the endpoints, and maximum principles  apriori are not available.
In Section 2 we will prove:

\begin{Theorem}
If $u\in C^{0} (M) \cap C^2 (M-S)$ is a non-negative $f$-invariant solution of Equation (\ref{q}), then $u\equiv 0$ or $u>0$. 
\end{Theorem}

We denote by
\[2^{*}_f(s):=\frac{2(n-k(f)-s)}{(n-k(f)-2)},\] 

\noindent
and note that if $k(f) \geq 1$ then $2^{*}_f(s) > 2^* (s)$. In Section 4 we will study the regularity of $f$-invariant 
solutions to Equation (\ref{q})
when $q\in (2, 2^*_f (s)$. In particular we will prove that if $u$ is a nonnegative $f$-invariant weak solution then
$u\in C^0 (M) \cap C^2 (M-S)$, so the previous theorem applies.

In Section 5 we will consider the functional $Q:H^2_{1,f}(M)-\{0\}\longrightarrow \R$ defined by
\[Q(u)=\frac{\int_M |\nabla u(x)|^2_g+K(x)u^2(x)dv_g}{\Big(\int_M\frac{|u(x)|^{q}}{d^s(x)}dv_g\Big)^{\frac{2}{q}}},\]
whose critical points are $f$-invariant, weak solutions of Equation (\ref{q}). We will prove the
existence of a positive  $f$-invariant solution that minimizes $Q$:

\begin{Theorem}\label{minimal}  Let $f$ be a proper isoparametric function of $(M,g)$, $s\in (0,2)$ and let 
$q\in (2,2^*_f (s) )$ .  Let $K\in C^0_f(M)=C^{0}(M)\cap S_f$ such that $\Delta_g+K$ is a coercive operator. 
Then,  there exist a positive  $f-$invariant weak solutions of Equation (\ref{q}).

\end{Theorem}

In Section 4 we will prove the following regularity result for solutions of Equation (\ref{q}):

\begin{Theorem}	Assume that $s\in (0,2)$ and $q\in (2, 2^*_f (s))$.
Any weak solution $u$ of Equation \eqref{q}  belongs to $C^{0,\alpha}(M)$ with $\alpha\in (0,\min(1,2-s))$, and  to $C_{loc}^{1,\beta}(M-S)$ with $\beta\in (0,1)$. 
Moreover, if $K\in C^{0,\delta}_f(M)=C^{0,\delta}(M)\cap S_f$ for some $\delta\in (0,1)$ then any non-negative weak solution of Equation \eqref{q} belongs to  $C^{2,\delta}_{loc}(M-S)$.
\end{Theorem}

Next we will see that using the mountain pass theorem of Ambrosetti and Rabinowitz in \cite{Ambrosetti-R} we can prove the existence of infinite solutions. Moreover,
we will prove that there is a  sequence of solutions $u_i$ which is not $C^0$-bounded, $ \lim_{i\rightarrow \infty} \max (u_i )=  \infty$.

\begin{Theorem}\label{mountain pass}  Let $f$ be a proper isoparametric function of $(M,g)$.  Let $K\in C^{0,\delta}_f(M)$ such that $\Delta_g+K$ is a coercive operator. 
Then,  for any $s\in(0,2)$, $q\in (2,2^*_f (s))$,  there exist infinitely many $f-$invariant solutions of Equation (\ref{q}). Moreover, there is a sequence of solutions
$u_i$ such that  $\int_M\frac{|u_i(x)|^{q}}{d^s(x)}dv_g  \rightarrow \infty$. 
\end{Theorem}

\begin{Remark} In the case $s=0$ the equation does not have a singularity. It corresponds to Yamabe type equations which have been extensively studied 
in the case of Riemannian manifolds. Although most of our proofs work without any change in the case $s=0$ we do not include this case 
for simplicity and because this case has been studied before. In particular  the previous result in the case $s=0$  was proved by J. Julio-Batalla in \cite{Batalla}.
\end{Remark}

In Section 6 we will prove that there is a $C^0$-bound for all positive solutions of Equation (\ref{q}) and we obtain:

\begin{Corollary}\label{nodalsolutions} Given a proper isoparametric function on $(M,g)$, $s\in (0,2)$, $q\in (2,2^*_f (s))$, if
$K\in C^{0,\alpha}_f (M)$ is such that $\Delta_g +K$ is coercive, then
there exist infinitely many nodal solutions, i.e., solutions that change sign, to the Equation \eqref{q}.
	\end{Corollary}

\begin{Remark} The most interesting case in our existence results is when $k(f) \geq 1$, since we cover in this situation 
supercritical exponents $q\in (2^* (s) , 2^*_f (s))$. 
Similar ideas to reduce  equations  to functions which display certain particular symmetries and consider compactness of 
embeddings of Sobolev spaces
have been considered for instance in \cite{Farkas, H, Mesmar}. 
There are no results on existence of solutions in this supercritical range in the general case. Similarly, we assume that
$K$ satisfies the condition that $\Delta_g + K$ is coercive. This is necessary to apply variational techniques and it would
be very interesting to study the situation when this condition is relaxed. 
\end{Remark}

\begin{Remark}
It is a very interesting problem to consider the limit case $q=2^*_f (s)$. This would be the critical case when restricted to
$f$-invariant functions. This critical case was consider in \cite{Mesmar} in the presence of symmetries. Different techniques used
in the critical situations can be found for instance  in \cite{FarkasFaraci, Jaber}.
\end{Remark}

The article is organized as follow. In Section 2 we exploit the fact that $f$-invariant solutions are characterized by solutions to an 
ordinary differential equation 
to prove Theorem 1.1 assuming that the solution is regular. 
We will also prove that positive invariant solutions must have a local maximum at the focal submanifolds. In Section 3 we consider the
spaces of $f$-invariant functions and study the corresponding embeddings. In Section 4 we consider the regularity of solutions
of Equation (\ref{q}). We use these results in Section 5 to prove the existence of a positive solution of Equation (\ref{q}). In Section
6 we prove that there is a $C^0$-bound for invariant positive solutions. Finally in Section 7 we prove Theorem 1.4.

	\section{Isoparametric functions}

Let  $(M,g)$ be a closed Riemannian manifold of dimension $n$.
Let $f:M\longrightarrow [t_0,t_1]$ be a proper isoparametric function, as in the introduction.
We let $M_0$ and $M_1$ be the focal submanifolds of $f$ and call $d_i= \dim(M_i)$.  We
call $S=M_0 \cup M_1$.
Since $f$ is proper we 
have that $d_i  \leq n-2$.

As in the introduction we denote by $D$ the distance between $M_0$ and $M_1$ and by
${\bf d} : M \rightarrow [0,D]$ the distance to $M_0$.  We write an $f$-invariant function 
$u$ on $M$ as
	
	\[u(x)=\varphi \circ {\bf d} (x),\]
	
\noindent 
where $\varphi :[0,D] \rightarrow \R$. It is well-known (see \cite{Betancourt} for details) that

	\begin{equation}\label{Laplacian0}
	\Delta_g u  (x) = -\varphi ''  ({\bf d} (x)) |\nabla {\bf d(x)}|^2+ \varphi' ({\bf d} (x))\Delta_g{\bf d}(x). 
	\end{equation}
	\noindent
\noindent
We also have that for $x\in M-S$

 \[\Delta_g  {\bf d}  (x)=\frac{b'(f(x))+2a(f(x))}{2\sqrt{b(f(x))}}=m({\bf d}(x)),\]

\noindent
where $b$ and $a$ are the functions that appear in Equations \eqref{Cgradient} and \eqref{Claplacian},  respectively. And $m(t)$  denotes the (constant) mean curvature of the submanifold ${\bf d}^{-1} (t)$ with respect to the normal vector field $N=\nabla f/\sqrt{b(f})$,
 $m:(0,D)\longrightarrow \R$.  See again \cite{Betancourt} for further details.

	Since $|\nabla {\bf d}|=1$, it follows  from Equation \eqref{Laplacian0}  that

	\begin{equation}\label{Laplacian}
	\Delta_g u  (x) = -\varphi ''  ({\bf d} (x)) +m( {\bf d} (x) ) \varphi '({\bf d} (x)) 
	\end{equation}

\bigskip	

	In this way,  equations involving the Laplacian are reduced, on the space of
	$f$-invariant functions, to ordinary differential equations. 
	
\bigskip

	The mean curvature of regular level sets approaches $\pm \infty$ near the focal submanifolds. The asymptotic behavior
	is well understood (see \cite{Betancourt} and \cite{Ge-Tang2}):
	
	\begin{equation}\label{asymptoticM_0}m(t) = -\frac{n-d_0 -1}{t} + A(t)
	\end{equation}
	and 
	\begin{equation}\label{asymptoticM_1}m(t) = \frac{n-d_1 -1}{D-t} + \tilde{A}(t)
	\end{equation}
	\noindent
	where $\lim_{t\rightarrow 0^+} A(t) =0$  and $\lim_{t\rightarrow D^-} \tilde{A}(t) =0$. Recall that   $n-d_i -1 \geq 1$ ($i=0,1$) since we are assuming that
	the isoparametric function is proper.  
	
	It follows from the Formula (\ref{Laplacian}) that $u=\varphi\circ {\bf d}$ is a solution of Equation \eqref{q} if and only if $\varphi$ satisfies:
	
	\begin{equation}\label{ODE1}
	\varphi '' (t) - m(t) \varphi ' (t) -{\bf k} (t) \varphi (t) +  \frac{\varphi^{q-1} (t)}{(d(t))^s} =0 ,
	\end{equation}
	
	\noindent
	where, ${\bf k}, d : [0,D] \rightarrow \R$, $K(x)= {\bf k}  ({\bf d}(x))$, $d(t)$ denotes the distance of ${\bf d}^{-1} (t)$  to 
$S= M_0 \cup M_1$:
	$d(t)=t$ if $t\in [0,D/2]$ and $d(t)=D-t$ if $t\in [D/2,D]$.

\begin{Remark}\label{Metric} Near a point  $x$ in any of the focal submanifolds $M_0$ or $M_1$ we will 
consider Fermi coordinates around $x$, relative to the focal submanifold. Assuming for instance that
$x\in M_0$ we consider first geodesic coordinates $(x_1,\dots ,x_{d_0} )$  around $x$ in $M_0$ and then 
geodesic coordinates $z=(x_{ d_0 +1},\dots ,x_n )$ in the directions normal to $M_0$.  Then 
in the coordinates $(x_1,\dots , x_n )$, $\| z \| = \| ( x_{d_0 +1},\dots ,x_n ) \|$ 
 is the distance of the point to $M_0$. 
An $f$-invariant function $u$ is written as $u = \varphi (\| z \| )$. 

It will be useful to consider an auxiliary  Riemannian metric $g_*$  on the $z$-ball of radius $\delta >0$,
$B_{\delta} := \{  z=(x_{ d_0 +1},\dots ,x_n ) : \| z \| < \delta  \}$. We let $g_0$ be the round metric
on the $(n-d_0 -1)$-sphere and consider $g_*$ of the form $g_* = dt^2 + r(t)^2 g_0 $ for some function
function $r:[0,\delta ) \rightarrow \R_{\geq 0}$, $r(0)=0$. For such metric
the mean curvature of the sphere of radius $t \in (0, \delta)$ is $-\frac{(n-d_0 -1)r' (t)}{r(t)}$.
By a simple computation, using  the asymptotic behavior of $m(t)$ near 0  (\ref{asymptoticM_0}), we 
 see that we can pick the function $r$ so that $\frac{(n-d_0 -1)r' (t)}{r(t)} = m(t)$. Moreover, 
$r$ can be chosen of the form $r(t) = t + o(t)$ so that $g_*$ is close to the Euclidean metric near the origin. 
Then it follows that for	an $f$-invariant function $u$, identified with a radial function on the
$z$-ball $B_{\delta}$,  we have
		
		$$\Delta_g u = \Delta_{g_*} u.$$

	\end{Remark}
	
	\bigskip
	
	Next we show that the maximum principle works in this situation:

{\bf Theorem 1.1}: If $u=\varphi \circ {\bf d} \in C^{0} (M) \cap C^2 (M- S )$ is an $f$-invariant non-negative solution of Equation \eqref{q} then
$\varphi \equiv 0$ or $\varphi >0$.

	\begin{proof} Let us assume that $\varphi$ is not identically zero.  If $\varphi (t)=0$ for some $t\in (0,D)$ then we must also have that $\varphi ' (t)=0$. But then, by the
		uniqueness of solutions with given initial conditions we must have that $\varphi \equiv 0$. 
		Then we can assume that $\varphi (t)>0$ for all $t\in (0,D)$. Thus, it suffices to show  that $\varphi (0)>0$ and
		$\varphi (D) >0$. Both cases are equivalent; we will  prove that  $\varphi (0)>0$. 
		Assume that $\varphi (0) =0$, so in particular $\varphi$ has a local minimum at $0$. 
		For $t\in (0,D/2)$ let $h(t) =-{\bf k} (t) +\frac{ (\varphi (t) )^{q-2} }{t^s}$. We have that 
		
	\[\varphi '' (t) -m(t) \varphi ' (t) + h(t) \varphi (t)=0\]
		\noindent for all $t\in (0,D/2)$.
		For $\alpha >0$ let
		
		\[v(t) = e^{-\alpha t} \varphi (t).\]
		
		\noindent
		Then, $v(0)=0$ and $v(t)>0$ for all $t\in (0,D/2)$.
		
		\medskip
		
		We have that $\varphi '(t) = e^{\alpha t} v'(t) +\alpha  e^{\alpha t} v(t)$ and
		$\varphi '' (t) =e^{\alpha t} v''(t) + 2 \alpha e^{\alpha t} v'(t) + \alpha^2 e^{\alpha t} v(t)$.

		
		\noindent
		Therefore,
		
		\[\varphi '' (t) - m(t) \varphi ' (t) + h(t) \varphi (t) = e^{\alpha t} \big[  v''(t) + (2\alpha -m(t)  ) v'(t)  + (\alpha^2 -\alpha m(t) + h(t)) v(t) \big] ,\]
		
		\noindent
		and then, 
		
		$$ v''(t) + \big(2\alpha -m(t)  \big) v'(t)  + \big(\alpha^2 - \alpha m(t) + h(t)\big) v(t)  = 0  .$$

		Note that $-m(t)$ is positive close to $t=0$, hence  it is bounded from below.  Also $h(t)$ is bounded below. Therefore, we can pick 
		$\alpha >0$ large enough so that $ 2\alpha -m(t) >0 $ and $\alpha^2 - \alpha m(t) + h(t) >0$ for all $t\in (0,D/2)$. 
		
		If at some $t_0 \in (0,D/2)$ we have that $v'(t_0 )=0$ we then have that $v'' (t_0 )< 0$. This implies that there is at most one critical
		point of $v$ in $(0,D/2)$, which would be a local maximum. Hence,  we can pick $t_1 \in (0,D/2)$ such that $v'(t) >0$ for
		all $t\in (0,t_1 )$. By the asymptotic behavior of $m$ at 0 (see Equation \eqref{asymptoticM_0}) we can pick $t_2 \in (0,t_1)$ such that $2\alpha- m(t) > \frac{1}{t}$ for
		all $t\in (0,t_2 )$.
		
		We have  that 
		
		\[v'' (t) +\frac{1}{t} v' (t) <  v''(t) + (2\alpha -m(t)  ) v'(t) <0.\] 
		\noindent Therefore, $ ( tv'(t) )'= t(v''(t) +(1/t) v'(t) )<0$ for all
		$t\in (0,t_2 )$. 
		
		This implies that for any $t\in (0,t_2 )$ we have that $tv'(t) > t_2 v' (t_2 ) =C >0$. Then $v' (t) > (1/t) C $. It follows
		that $\lim_{t\rightarrow 0^+} v(t) = -\infty$. Then, also $\lim_{t\rightarrow 0^+} \varphi (t) = -\infty$, which is a
		contradiction since we were assuming that $\varphi$ is a non-negative function.  Then we must have that
		$\varphi (0)>0$.
		
	\end{proof}

	When an $f$-invariant solution $u$  of Equation \eqref{q} is positive and  $u\in C^{0,\alpha} (M)$ for some $\alpha \in (0,1)$ we can prove that
	$0$ is actually a local maximum.

	\begin{Lemma}\label{Cero} For any $f$-invariant positive solution 
		$u=\varphi \circ {\bf d} \in C^{0,\alpha} (M) \cap C^2 (M- S )$ of Equation \eqref{q}, $\varphi$ has a local maximum at 0.

	\end{Lemma}
	
	\begin{proof}
		Since $\varphi$ solves Equation \eqref{ODE1}, $d(0)=0$  and $\varphi (0) >0$ there exists $t_0 >0$ such that $\varphi ''(t) - m(t) \varphi ' (t) <0$ for
		all $t\in (0,t_0 )$.  
		
		This implies that any critical point of $\varphi$ in  $(0,t_0 )$ is a local maximum. By taking $t_0$ smaller if necessary we can assume that
		$\varphi '$ does not vanish in $(0,t_0 )$. Therefore, if $0$ is not a local maximum we must have that $\varphi ' (t) >0$ for all
		$t\in (0, t_0 )$ (and 0 is a local minimum). Assume that this is the case.

		By the asymptotic behavior of $m$ (see Equation \eqref{asymptoticM_0}), and taking 
		$\beta < n-d_0 -1$ we may assume that 
		$\varphi '' (t) + \frac{ \beta }{t} \varphi ' (t) < \varphi ''(t) - m(t) \varphi ' (t)   <0$  for any $t\in (0,t_0 )$  (taking a smaller $t_0$ if necessary). 
		Recall that we are assuming that the isoparametric function $f$ is proper, so $n-d_0 -1 \geq 1$.
		
		We have  that $(t^{\beta} \varphi ' (t) )'  = t^{\beta} \left( \varphi '' (t) + \frac{ \beta }{t} \varphi ' (t) \right)  <0$. 
		It follows then that for any $t\in (0,t_0 )$ we have \[t^{\beta} \varphi ' (t) > t_0^{\beta} \varphi ' (t_0 )=C >0,\]
		or equivalently  $\varphi ' (t) >C t^{-\beta}$. 
		
		If $n-d_0 -1 >1$ we can take $\beta >1$ and we reach a contradiction since  the integral of $\varphi '$ in $(0,t_0 )$  would be $+\infty$.

		If $n-d_0 -1 =1$ we can take any $\beta <1$. In particular, if  $u=\varphi \circ {\bf d} \in C^{0,\alpha} (M)$ we can take $1-\alpha <\beta <1$.
		For any fixed $t\in (0,t_0 )$ and any $s$ such that $0<s<t$ we have that for some positive constant $R$

		$$\frac{\varphi (t) - \varphi (s)}{(t-s)^{\alpha} } \leq R. $$
		
		On the other hand, there exists $x\in (s,t)$ such that $\varphi ' (x) (t-s) = \varphi (t) -\varphi (s)$. Since $\varphi ''<0$ in  $(0,t_0 )$ we have that
		$\varphi ' (x) > \varphi ' (t)$. Then we have that
		
		$$\frac{ \varphi '(t) (t-s) }{(t-s)^{\alpha}} <R ,$$
		
		\noindent
		or
		$\varphi' (t) < R  (t-s)^{\alpha -1}$. Since this is valid for all $s$ we can take the limit as $s\rightarrow 0$ and we have that 
		$\varphi ' (t) \leq R t^{\alpha -1}$.

		Therefore we have that for all $t\in (0,t_0 )$

		$$\frac{C}{t^{\beta}} < \varphi ' (t)  \leq \frac{R}{t^{1-\alpha}} .$$

		Then we would have that for all $t$ small
		
		$$\frac{1}{t^{\beta - (1-\alpha)}} < \frac{R}{C} ,$$
		
		\noindent
		which is a contradiction.
		
	\end{proof}

We will need the following result in Section 6.

\begin{Lemma}\label{U} Assume that there is a sequence $u_i = \varphi_i \circ {\bf d}
\in C^{0,\alpha} (M) \cap C^2 (M- S )$ of $f$-invariant positive  solutions of
Equation \eqref{q}  such that $\max \varphi_i = \varphi_i (t_i ) \rightarrow \infty$ with 
$\lim_{i\rightarrow \infty}  t_i =0$. Then
$t_i =0$ for $i$ large enough.

\end{Lemma}	
	
\begin{proof}
Let $c>0$ such that $| {\bf k} | \leq c$. Let $r>0$ small such that $-c +r^{-s} >0$
and $m(t) <0$ for all $t\in (0,r)$. It follows that
if $t \in (0,r)$ is a local minimum of $\varphi_i$ then $\varphi_i (t) \leq 1$. We know from the
previous lemma that 0  is a local maximum of $\varphi_i$.  Assume that for some $i$ there
exists a local minimum $t_0 \in (0,r)$ of $\varphi_i$ and let $\delta >0$ such that $t_0+\delta \leq r$
and $\varphi_i$ is increasing in $(t_0 ,t_0 +\delta)$. Then we have that $\varphi_i '' (t) \leq c \varphi_i (t)$
for any $t\in (t_0 , t_0
+\delta )$. 
Since $\varphi_i' (t_0)=0$ it follows that $\varphi_i ' (t) \leq cr \varphi_i  (t)$ for any $t\in (t_0 , t_0 + \delta )$.
Since $\varphi_i (t_0 ) \leq 1$ this implies that there is a uniform (independent of $i$) bound for 
$\varphi_i (t_0 +\delta )$. Since $\lim_{i\rightarrow \infty}  \varphi_i (t_i ) = \infty$ this implies that $t_i =0$ for $i$ large enough.
\end{proof}
	

\section{Inclusions of spaces of $f$-invariant functions}\label{R-Ksection}

Let $f:M\longrightarrow \R$ be an isoparametric function and $d$ the distance function to the submanifold $S=M_0\cup M_1$. 
Let $q\geq 1$ and $s\in[0,2]$. The weighted Lebesgue space $L^q(M,d^{-s})$ is  the set of functions $u:M\longrightarrow\R$  that satisfy 

	\[\int_M \frac{|u(x)|^q}{d^s(x)}dv_g<\infty.\]
	 \noindent
	 We define a norm on  $L^q(M,d^{-s})$ by setting \[\|u\|_{L^q(M,d^{-s})}=\Big(\int_M \frac{|u(x)|^q}{d^{s}(x)}dv_g\Big)^{\frac{1}{q}}.\]
	 Hence, $L^q(M,d^{-s})$ is a normed space.

	Let $x_{0}\in M$ and let us consider the distance function to $x_0$, $\nu(x)=d_g(x_0,x)$. It is well known that for any $q\leq 2^*(s)$, the inclusion map of the Sobolev space
	 $H^2_1(M)$ into $L^q(M,\nu^{-s})$ is continuous. Moreover, if  $q< 2^*(s)$,  this embedding is compact. The
conditions under which the inclusion holds can be improved if we restrict to $f$-invariant functions. 
Recall that we denote by $k(f)$ the minimum of the dimension of the focal submanifolds of $f$.  
	Let  
	\[2^{*}_f(s):=\frac{2(n-k(f)-s)}{(n-k(f)-2)}.\] 
	 Then the following proposition holds:

	\begin{Proposition}\label{R-Kgeneral} Let $s\in (0,2)$. 
		\begin{itemize}
			\item[i)]   If $n-k(f)>2$ and $q\leq 2^{*}_f(s)$, then the inclusion map  of  $H^2_{1,f}(M)$ into $L^q(M,d^{-s})$ is continuous. Moreover, the embedding  is compact if the inequality is strict.

			\item[ii)] If $n-k(f)\leq 2$  then the inclusion of  $H^2_{1,f}(M)$ into  $L^q(M,d^{-s})$ is compact for any $q\geq 1+\frac{s}{2}$.
		
	\end{itemize} 
		
	\end{Proposition}

	To prove Proposition \ref{R-Kgeneral}, we will make use of the following lemmas. Lemma \ref{L1}  generalizes  the classical Hardy Inequality to the case where  the singular set is given by  the focal submanifolds.
	
	\begin{Lemma}\label{L1} There exists  a constant $C>0$ such that for  any $u\in H^2_1(M)$ it holds 
		\begin{equation}\label{Ineq-Hardy-variedad}\int_M\frac{|u|^2}{d^2(x)}dv_g\leq C\|u\|_{H^2_1(M)}^2.\end{equation}
		\end{Lemma}

		This lemma implies  that the embedding of $H^2_1(M)$ into $L^2(M,d^{-2})$ is continuous.
		
		\begin{proof}  Using Fermi coordinate systems, it can be   shown  that for  any point $x\in M$ there exists a  chart $(W_x,\varphi_x)$ that satisfies the following properties:  $(a)$ $x\in W_x$; $(b)$ $\varphi_x(W_x)=U_x\times V_x$, where $U_x\subset \R^{k(x)}$ and   $V_x\subset \R^{n-k(x)}$ are open sets with smooth boundaries (here, $k(x)$ denotes the dimension of the connected component of $f^{-1}\left(f(x)\right)$ to which $x$ belongs);  $(c)$ for any $\overline{x}\in W_x$, $U_x\times \pi(\varphi(\overline{x}))\subseteq \varphi\left(f^{-1}(f(\overline{x}))\right)\cap W_x$ (for $(y,z)\in \R^{k(x)}\times \R^{n-k(x)}$,   $\pi(y,z)=z$);
			$(d)$  there exists  $a_x>0$ such that $a_x^{-1}dv_{g^n_e}\leq dv_g\leq a_x dv_{g^n_e}$,
			where $dv_{g^n_E}$ is the volume element induced by Euclidean metric of $\R^n$. 
			
		Let $u\in S_f$. Properties $(b)$ and  $(c)$ imply that the real valued function  $u_x:V_x:\longrightarrow \R$ defined  by \[u_x(z)=u\left(\varphi_x^{-1}(y,z)\right)\]
			is well defined.  On the other hand,  by property $(d)$  we can compare the  weighted Lebesgue norms (for any weight $\omega$)  and the  $H^2_1$-norms of the functions $u_x$ and $u|_{W_x}$. In particular, for  the weight $d^{-2}$ there exist positive constants $A_x$ , $B_x$, and $C_x$ such that 
			
			\begin{equation}\label{Ineq 1}
			A_x\|u_x\|_{L^r\left(V_x,|y|^{-2}\right)}\leq \|u|_{W_x}\|_{L^r\left(W_x, d^{-2}\right)}\leq B_x\|u_x\|_{L^r\left(V_x,|y|^{-2}\right)}
			\end{equation}
			
			\begin{equation}\label{Ineq 2}
			\|\nabla u_x\|_{L^r\left(V_x\right)}\leq C_x\|\nabla u|_{W_x}\|_{L^r\left(W_x\right)}
			\end{equation}
			
			Since $M$ is compact, there exists a finite collection of points $x_i$ $(i=1,\dots,m)$ such that $M=\bigcup_{i=1}^m W_{x_i}$. In order to simplify the notation we denote this family of charts by $\{(W_i,\varphi_i)\}_{i=1,\ldots,m}$. 
			
			The classical Hardy inequality says that  any $\psi\in C^{\infty}_c(\R^n)$ satisfies
			\begin{equation}\label{HardyEuclidea}
			\int_{\R^n} \frac{\psi^2(x)}{|x|^2}dv_{g^n_E}\leq \frac{4}{n-2}\int_{\R^n}|\nabla \psi(x)|^2dv_{g^n_E}.
			\end{equation}
			
			Therefore, from \eqref{Ineq 1}, \eqref{Ineq 2}, and  \eqref{HardyEuclidea}  it follows that
			
			\begin{align}
			\int_{W_i}\frac{u|_{W_i}^2(x)}{d^2(x)}dv_g&\leq D\int_{U_i\times V_i}\frac{u\circ {\varphi_i}^{-1}(y,z)}{d^2(\varphi^{-1}_i(y,z))}dxdy=D.vol(U_i)\int_{V_i}\frac{u_i^2(y)}{|y|^2}dy\nonumber \\&\leq \overline{D}\int_{V_i}|\nabla u_i|^2dy\leq \tilde{D} \int_{W_i}|\nabla u|_{W_i}|^2dv_g\nonumber. 
			\end{align}
			
			This  implies straightforward Inequality \eqref{Ineq-Hardy-variedad}.
			
			\end{proof}

	 The non-singular case ($s=0$) in Proposition \ref{R-Kgeneral} was addressed in  (\cite{H}, Lemma 6.1). For the reader's convenience, we restate the result below. We refer to \cite{H} for the proof.

		\begin{Lemma}\label{L2} Let $q\geq 1$.
			
			\begin{itemize}
				\item[i)] If $n-k(f)>2$ and $q\leq 2^{*}_f(0)$ then the inclusion map of  $H^2_{1,f}(M)$ into $L^q(M)$ is continuous. The inclusion is compact if $q< 2^{*}_f(0)$.
				
				\item[ii)] If $n-k(f)\leq 2$, then  the inclusion map of  $H^2_{1,f}(M)$ into $L^q(M)$  is  compact  for any $q\geq 1$.

	\end{itemize}

			\end{Lemma}

\vspace{0.2cm}

\begin{proof}[Proof of Proposition \ref{R-Kgeneral}] We begin by proving the continuity  and then the compactness of the inclusion map.
	
\vspace{0.2cm}

	{\bf{ Continuity of the inclusion map:}}
	 Lemma \ref{L1} and Lemma \ref{L2}  cover the cases when either $s=0$ or $s=2$, respectively. Therefore, we  assume that $s\in (0,2)$.
	
	\begin{itemize}
	\item[(a)] Assume that $n-k(f)>2$ and $q\leq 2^{*}_f(s)$.  
	
	\begin{itemize}
		
	\item[(a.1)]	First, will we show that the inclusion map  
	\[i: H^2_{1,f}(M)\longrightarrow L^{2^{*}_f(s)}(M,d^{-s} )\] 
	is continuous.   Indeed,

	\begin{align}\|u\|_{L^{2^{*}_f(s)}(M,d^{-s})}^{2^{*}_f(s)}&=\int_M \frac{|u|^{2^{*}_f(s)}}{d^s(x)}dv_g=
	\int_M |u|^{2^{*}_f(s)-s}\Big(\frac{u}{d(x)}\Big)^sdv_g\nonumber\\
	&\leq\Big( \int_M |u|^{\frac{2(2^{*}_f(s)-s)}{2-s}}dv_g \Big)^{\frac{2-s}{2}} \Big( \int_M\frac{|u|^2}{d^2(x)}dv_g \Big)^{\frac{s}{2}}.  \nonumber
	\end{align}
	
	Since,
	$2^{*}_f(s)-s=\frac{2(n-k(f)-s)}{n-k(f)-2}-s=\frac{(2-s)\left(n-k(f)
		\right)}{n-k(f)-2}$ it follows that
	
	\[\|u\|_{L^{2^{*}_f(s)}(M,d^{-s})}^{2^{*}_f(s)}\leq \Big( \int_M |u|^{\frac{2(n-k(f))}{n-k(f)-2}}dv_g \Big)^{\frac{2-s}{2}} \Big( \int_M\frac{|u|^2}{d^2(x)}dv_g \Big)^{\frac{s}{2}}. \]
	Therefore,

\begin{align}\|u\|_{L^{2^{*}_f(s)}(M,d^{-s})}&\leq \Big( \int_M |u|^{2^{*}_f(0)}dv_g \Big)^{\frac{2-s}{2.2^{*}_f(s)}} \Big( \int_M\frac{|u|^2}{d^2(x)}dv_g \Big)^{\frac{s}{2.2^{*}_f(s)}} \nonumber\\
&=\Big(\|u\|_{L^{2^{*}_f(0)}(M)}\Big)^{\frac{(2-s)2^{*}_f(0)}{2.2^{*}_f(s)}}
\Big( \int_M\frac{|u|^2}{d^2(x)}dv_g \Big)^{\frac{s}{2.2^{*}_f(s)}}\nonumber\\
&=\Big(\|u\|_{L^{2^{*}_f(0)}(M)}\Big)^{\frac{(2-s)(n-k(f))}{(n-k(f)-2)2^{*}_f(s)}}\Big( \int_M\frac{|u|^2}{d^2(x)}dv_g \Big)^{\frac{s}{2.2^{*}_f(s)}}.\nonumber
\end{align}

By Lemma \ref{L1} and Lemma \ref{L2} we have that

\begin{align}\|u\|_{L^{2^{*}_f(s)}(M,d^{-s} )}&\leq \Big(C_1 \|u\|_{H^2_1(M)} \Big)^{\frac{(2-s)(n-k(f))}{(n-k(f)-2)2^{*}_f(s)}} \Big(C_2 \|u\|_{H^2_1(M)}\Big)^{\frac{s}{2^{*}_f(s)}}\nonumber\\
&=\tilde{C} \|u\|_{H^2_1(M)}^{\big(\frac{(2-s)(n-k(f))}{(n-k(f)-2)}+s \big)\frac{1}{2^{*}_f(s)}}=\tilde{C}\|u\|_{H^2_1(M)}.\nonumber
\end{align}

\item[(a.2)] Let us assume that $q<2^{*}_f(s)$.  It follows from  H\"{o}lder inequality that

\begin{align}\|u\|_{L^q(M,d^{-s}  )}^q&=\int_M |u|^{q}\frac{1}{d^s(x)}dv_g\nonumber\\ &\leq  \Big(\int_M \frac{|u|^{2^{*}_f(s)}}{d^s(x)}dv_g\Big)^{\frac{q}{2^{*}_f(s)}}\Big( \int_M \frac{1}{d^s(x)}dv_g  \Big)^{1-\frac{q}{2^{*}_f(s)}}.\nonumber\end{align}

Hence, using (a.1) we have that 

 \begin{align}\|u\|_{L^q(M,d^{-s})}&\leq C\|u\|_{L^{2^{*}_f(s)}(M,d^{-s} )}\leq \tilde{C} \|u\|_{H^2_1(M)},
 \nonumber
 \end{align}

 \noindent
 which completes the proof in this case.

\end{itemize}

\item[(b)] Assume that $n-k(f)\leq  2$ and $q\geq 1+\frac{s}{2}$. Using  the H\"{o}lder inequality we have that
\begin{equation}\label{desigconHolder}\int_M \frac{|u|^{q}}{d^s(x)}dv_g\leq \Big( \int_M |u|^{\frac{2(q-s)}{2-s}}dv_g \Big)^{\frac{2-s}{2}} \Big( \int_M\frac{|u|^2}{d^2(x)}dv_g \Big)^{\frac{s}{2}}. 
\end{equation}
Therefore, using  Lemma \ref{L1},
\begin{align}\|u\|_{L^{q}(M,d^{-s})}&\leq \Big( \int_M |u|^{\frac{2(q-s)}{2-s}}dv_g \Big)^{\frac{(2-s)}{2(q-s)}\frac{(q-s)}{q}} \Big( \int_M\frac{|u|^2}{d^2(x)}dv_g \Big)^{\frac{s}{2q}}.  \nonumber\\
&\leq C\Big(\|u\|_{L^{\frac{2(q-s)}{2-s}}(M)}\Big)^{\frac{q-s}{q}}\|u\|_{H^2_1(M)}^{\frac{s}{q}}.\nonumber
\end{align}

Since $ \frac{2(q-s)}{(2-s)}\geq 1$, then it follows from Lemma \ref{L2}  that 
\[\|u\|_{L^{q}(M,d^{-s})}\leq \tilde{C} \|u\|_{H^2_1(M)}^{\frac{q-s}{q}}\|u\|_{H^2_1(M)}^{\frac{s}{q}}=\tilde{C} \|u\|_{H^2_1(M)}.\]

which establishes the continuity of the inclusion map in this case.

\end{itemize}

\vspace{0.2cm}

{\bf Compactness of the inclusion map:}

\begin{itemize}

\item[(a)] Assume that $q<  2^{*}_f(s)$ and $n-k(f)>2$. Let $\{u_k\}_{k\in \N}$ be a bounded sequence in  $H^2_{1,f}(M)$. By the classical  Rellich-Kondrakov theorem the inclussion of 
	$H^2_{1,f}(M)$ into $L^2(M)$ is  a compact map. Therefore,  up to subsequence,  $\{u_{k}\}_{i\in \N}$  converges weakly in
$H^2_1(M)$ and  strongly in $L^2(M)$, to a function $u\in H^2_1(M)$.
	
	Given $T>0$, let $A^k_T$ be the subset of $M$ defined by $A^k_T:=\{x\in M:\ |u_{k}(x)-u(x)|<T\}$. We have that

\begin{align}\int_{M-A^k_T} \frac{|u_{k}(x)-u(x)|^q}{d^s(x)}dv_g&=\int_{M-A^k_T} |u_{k}(x)-u(x)|^{q-2^{*}_f(s)}\frac{|u_{k}(x)-u(x)|^{2^{*}_f(s)}}{d^s(x)}dv_g \nonumber\\
 &\leq \frac{1}{T^{2^{*}_f(s)-q}}  \int_{M} \frac{|u_{k}(x)-u(x)|^{2^{*}_f(s)}}{d^s(x)}dv_g\nonumber\\&\leq \frac{C}{T^{2^{*}_f(s)-q}} \|u_{k}-u\|_{H^2_1(M)}^{2^{*}_f(s)}\nonumber\\
 &\leq \frac{\tilde{C}}{T^{2^{*}_f(s)-q}}\longrightarrow_{T\to +\infty} 0\nonumber
 \end{align}

On the other hand, let   $\{h_{k}\}_{k\in \N}$ be the sequence defined by  \[h_{k}=\frac{|u_{k}-u|^q}{d^s(x)}\chi_{A^k_T},\]
where $\chi_{A^k_T}$ is the characteristic function of the set $A^k_T$.  The functions  $h_{k}$ are bounded  above by  $\frac{T^q}{d^s}$, which is integrable.  Since  $\{h_{k}\}_{k\in \N}$ converges a.e. to $0$,  the dominated convergence theorem implies that  
\[\lim_{k\to +\infty} \int _M h_{k}dv_g=0.\]
Hence,  given $\varepsilon>0$, for  $T$ and $k$  big enough we have that 
\[\|u_{k}-u\|_{L^q(M,d^{-s})}^q=\int_{M-A^k_T}\frac{|u_{k}(x)-u(x)|^q}{d^s(x)}dv_g+\int_{A^k_T}\frac{  |u_{k}(x)-u(x)|^q}{d^s(x)}dv_g\leq \varepsilon. \]
 This proves that $\{u_{k}\}_{k\in N}$ converges to $u$ in $L^q(M,d^{-s})$.

\vspace{0.2cm}

\item[(b)] Finally, we assume that $n-k(f)\leq 2$ and $q\geq 1+\frac{s}{2}$.  Let $\{u_k\}_{k\in \N}$ be  a  bounded sequence in $H^2_{1,f}(M)$. Since $\frac{2(q-s)}{2-s}\geq 1$,  it follows from  Lemma \ref{L2}  that there exists a function $u$  such that, up to a subsequence, $\{u_{k}\}_{k\in \N}$  converges strongly to $u$ in $L^{\frac{2(q-s)}{2-s}}(M)$.  

By Equation \eqref{desigconHolder} we have that
\[\|u_{k}-u\|_{L^q(M,d^{-s})}\leq \Big( \int_M |u-u_k|^{\frac{2(q-s)}{
		(2-s)}}dv_g \Big)^{\frac{2-s}{2q}} \Big( \int_M\frac{|u-u_k|^2}{d^2(x)}dv_g \Big)^{\frac{s}{2q}}. 
\]
Note that $\|u_{k}-u\|_{H^2_1(M)}$ is a bounded sequence, therefore  from Lemma \ref{L1} it follows that  

\begin{align}\|u_{k}-u\|_{L^q(M,d^{-s})}&\leq C\Big(\|u_{k}-u\|_{L^{\frac{2(q-s)}{2-s}}(M)}\Big)^{\frac{q-s}{q}}\|u_{k}-u\|_{H^2_1(M)}^\frac{s}{q}\longrightarrow_{k\to+\infty} 0.\nonumber
\end{align}

\end{itemize}

\end{proof}

\section{Regularity of solutions.}\label{regularity}

Let $f$ be a proper  isoparametric function on a closed Riemannian manifold 
$(M,g)$ of dimension $n\geq 3$. In  this section we address the regularity properties of $f$-invariant weak solutions to  the equation 
(\ref{q}) when
 $K\in C^0_f(M)$ and  $q<2_f^*(s)$. The treatment is similar to the results in \cite{Ghoussoub-Robert, Jaber, Mesmar}, 
and we refer to these articles for more details on the proofs.

Recall that $f$  proper  means that ${\bf n} = \min \{ n-d_0 , n-d_1 \} \geq 2 $. 

By considering Fermi coordinates around points in focal submanifolds we can see that $d^{-s} \in L^p (M)$ for any
$p<{\bf n}/s$.

\begin{Remark}
Any weak solution of Equation \eqref{q} belongs to $L^p(M)$ for any $p>1$. 
This follows as in 
(\cite{Jaber}, Step 4 in the proof of Theorem 2),  following the same argument as in  (\cite{Ghoussoub-Robert}, Proposition 8.1).
\end{Remark}

\bigskip

It is easy to prove regularity of solutions away from the singularities:

\begin{Proposition}\label{regularityC1loc} Let $u\in H^2_1(M)$ be a weak solution of Equation \eqref{q}. Then $u\in C^{1,\beta}_{loc}(M-S)$ 
	for any $\beta\in (0,1)$.
\end{Proposition}

\begin{proof} For any $p>1$, we have that  $f(x,u)=\frac{u^{q-1}(x)}{d^{s}(x)}-K(x)u(x) \in L^p_{loc}(M- S)$. 
	Since $u$ satisfies the equation $\Delta_g u =f(x,u)$ we see in
	particular that for any $p>n$ we have that $u\in H^p_{2,loc} (M-S)$ by elliptic regularity. By the Sobolev embedding theorem we know that $H^p_{2,loc}(M-S)$ is continuously embedded in $C^{1,\beta}_{loc}(M-S)$ provided $\beta<1-\frac{n}{p}$. Hence, any $u\in H^2_1(M)$ that is a weak solution of Equation \eqref{q} belongs to $C^{1,\beta}_{loc}(M-S)$ 
for any  $\beta\in (0,1)$.

\end{proof}

We also have:

\begin{Proposition}\label{regularityC2loc} Let $u\in H^2_1(M)$ be a non-negative weak solution of Equation \eqref{q} with $K\in C^{0,\delta}(M)$ with $\delta\in (0,1)$. Then $u\in C^{2,\delta}_{loc}(M-S)$.
\end{Proposition}

\begin{proof}  By the previous proposition we have that $f(x,u)=\frac{u^{q-1}(x)}{d^s(x)}-K(x)u(x)\in C^{0,\delta}_{loc}(M-S)$. 
	Since $u$ is a weak solution of $\Delta_g u=f(x,u)$, by applying Schauder estimates, we obtain  that $u\in C^{2,\delta}_{loc}(M-S)$.
	
\end{proof}

Finally we have the following result:

\begin{Proposition}\label{regularityC0} Let $u\in H^2_1(M)$ be a weak solution of Equation \eqref{q}. Then $u\in C^{0,\alpha}(M)$ 
	for any  $\alpha\in (0,\min(1,2-s))$.
\end{Proposition}

\begin{proof} Let $u$ be a weak solution. First we show that  $u\in H^p_2(M)$ for some $p>1$. 
	
	Let us  consider any $1<p<\frac{\bf n}{s}$. We have that for any $l\geq 1$ the function 
	
\[h(x)=\frac{u^l(x)}{d^s(x)}\]
	belongs to $L^{p}(M)$. Indeed, let   $p<\tilde{p}<\frac{n}{s}$  and let $\bar{p}>1$ such that $1=\frac{p}{\tilde{p}}+\frac{p}{\bar{p}}$. Applying the H\"older inequality and the fact that  $d^{-s}(x)\in L^{\tilde{p}}(M)$ when $\tilde{p}<\frac{\bf n}{s}$, it follows that 
	\begin{align}
	\int_M\Big(\frac{u^l(x)}{d^s(x)}\Big)^pdv_g&\leq \Big(\int_M u^{l\bar{p}}(x)dv_g\Big)^{\frac{p}{\bar{p}}}.\Big(\int_M \frac{1}{d^{s\tilde{p}}(x)}dv_g\Big)^{\frac{p}{\tilde{p}}}\leq C\Big(\int_M u^{l\bar{p}}dv_g\Big)^{\frac{p}{\bar{p}}},\label{lq}
	\end{align}
	for some constant $C$. 
	
	As we mentioned in the previous Remark, we know that any weak solution of Equation \eqref{q} 
	belongs to $L^r(M)$ for any $r>1$. Therefore,  Inequality \eqref{lq} implies that $h$ is in $L^p(M)$. This implies that the function $f(x,u)=\frac{u^{q-1}(x)}{d^s(x)}-K(x)u(x)$ belongs to $L^p(M)$ for any $1<p<\frac{\bf n}{s}$. 
	
	From the previous propositions we only need to prove that for any point in the focal submanifolds there is a neighborhood $U$ of the
	point such that $u\in C^{0,\alpha}(U)$. We consider Fermi coordinates around the point. As $u$ is $f$-invariant it depends only on the
	coordinate $z=(x_{d_i+1},\dots,x_n )$  where $d_i=\dim(M_i)$. In the ball of dimension $n-d_i \geq {\bf n}$, $B_{n-d_i}$,  we have a metric $g_*$ close to the Euclidean metric near the
	origin such that  $u$ is a weak solution of $\Delta_{g_*}  u=f(x,u)$ (see Remark (\ref{Metric})). 
 It follows from elliptic regularity that $u\in H^p_2(B_{n-d_i})$, for some ball centered at 0 in the $z$-coordinates (for any $p\in (1,\frac{\bf n}{s}$)). 
	
	By the Sobolev embedding theorems we  have that $H^2_2(B_{n-d_i})$ is continuously embedded into $C^{0,\alpha}(B_{n-d_i})$, provide  $\alpha<2-\frac{n-d_i}{p}
	<2-\frac{\bf n}{p}$. Therefore, $u\in C^{0,\alpha}(M)$ for any  $\alpha\in (0,\min(1,2-s))$.
	
\end{proof}

\section{Minimizing solution}	
		Let us consider the functional $Q:H^2_{1,f}(M)-\{0\}\longrightarrow \R$ defined by

	\[Q(u)=\frac{\int_M |\nabla u(x)|^2_g+K(x)u^2(x)dv_g}{\Big(\int_M\frac{|u(x)|^{q}}{d^s(x)}dv_g\Big)^{\frac{2}{q}}}.\]
	
	Critical points of $Q$ are, up to a multiplication by a constant,   weak solutions to Equation \eqref{q}.  Let us denote the infimum of $Q$ over $H^2_{1,f}(M)-\{0\}$ by $Y(M,f,s)$.
	
\begin{Theorem}\label{minimizing}  
Let $f$ be a proper isoparametric function of $(M,g)$.  Assume that $s\in (0,2)$ and  $q<2^*_f (s)$. Assume also that $K\in C^{0,\delta}_f(M)$ for some $\delta\in (0,1)$.
Then there exists a positive (minimizing) weak solution  $u$ of Equation \eqref{q}.   Moreover, $u$ belongs to $C^{0,\alpha}(M) \cap  C^{2,\delta}_{loc}(M-S)$ 
with $\alpha\in (0,\min(1,2-s))$. 
\end{Theorem}

\begin{proof} It follows  from Proposition \ref{R-Kgeneral}  that the inclusion map of $H^2_{1,f}(M)$ into $L^{q}(M,d^{-s})$ is compact. Let $\{u_k\}_{k\in\N }$ be a 
minimizing sequence of $Q$. We can assume that $u_k \geq 0$  and  $\|u_k\|_{L^q (M,d^{-s}) }=1$ for any $k$. Then, $\{u_k\}$ must contain  a  subsequence that 
converges to a non-negative $f-$invariant minimizer function $u\neq 0$. Therefore, $Y(M,f,s)$ is attained  by $u$, which is a critical point of  $Q$. The regularity of the
 solution $u$ was proved in Section \ref{regularity} (see Propositions  \ref{regularityC1loc}, \ref{regularityC2loc} and \ref{regularityC0}). From Theorem 1.1  
it follows that $u$ is positive.
	 
	\end{proof}

\section{$C^0$-bound for positive $f$-invariant solutions}

Let $f: M \rightarrow \R$ be a proper isoparametric function. Let $K \in C^{0,\delta}  (M)$ be an $f$-invariant function.
We let $M_0$, $M_1$ be the focal submanifolds of $f$. We recall that the level sets of $f$ are at constant distance
to $M_0$ and $M_1$.  

In this section we will consider the equation

\begin{equation}\label{E}
\Delta u(x) +K(x) u(x) = \frac{u^q(x)}{d(x)^s},
\end{equation}

\noindent
with $s\in (0,2)$, $q<2_f^* (s)$ and prove that there is a $C^0$-bound for positive 
$f$-invariant solutions. 

\medskip

Recall that  ${\bf {d}} :M\rightarrow [0,D]$ denote the distance to
$M_0$. If $u$ is an $f$-invariant function, then $u= \varphi \circ {\bf{d}}$, for some function $\varphi: [0,D] 
\rightarrow \R$.   We let ${\bf k}: [0,D] \rightarrow \R$ be the function such that $K(x) ={\bf k}\left({\bf d} (x)\right)$.
Then $u$ solves the equation if and only if
$\varphi$ verifies:

\begin{equation}\label{ODE}
\varphi '' (t) -m(t) \varphi ' (t) -{\bf k} (t) \varphi (t) +  \frac{\varphi^q (t)}{d^s(t)} =0
\end{equation}

\noindent
where $m(t)$ denotes the mean curvature of ${\bf d}^{-1} (t)$ and $d(t)$ denotes the distance of ${\bf{d}}^{-1} (t)$ to
the focal submanifolds $S=M_0 \cup M_1 $: $d(t) = t $ if $t\in [0,D/2]$
and $d(t)=D-t$ if $t\in [D/2,D]$.

We prove:

\begin{Theorem}There is a uniform $C^0$ bound for non-negative $f$-invariant solutions of Equation (\ref{E}).

\end{Theorem}

\begin{proof}
	
	Assume that there is a sequence $\{u_i\} _{i\geq 1}$ of non-negative solutions of Equation (\ref{E})  such that 
	$a_i = max \  u_i  \rightarrow +\infty$. Let $p_i \in M$ such that
	$u_i (p_i )= a_i$. We can assume that the sequence $p_i$ converges to
	some $p\in M$.

	{\bf Case 1}: $p\in S$. We consider a normal neighborhood of $p$ in $S$ and then geodesic coordinates in the 
	directions normal to $S$. We have then a neighborhood of $p$ with coordinates $(x_1, \dots , x_k, x_{k+1}, \dots , x_n )$
	so that the distance of $x$ to $S$ is $\| (x_{k+1}, \dots ,x_n ) \|$. Since $u$ is $f$-invariant it only depends on the 
	variable $z =(x_{k+1},\dots , x_n ) $, and actually only on $\| z \|$.
	
	We can express $u_i$ as a function of the variable $z$ which verifies

	\begin{equation}\label{R}
	\Delta_{g_*} u_i (z)+K(z) u_i(z) = \frac{u_i^q(z)}{\| z \|^s},
	\end{equation}
	
	\noindent
	where $g_*$ is  the metric on the $z$-slice given by Remark \ref{Metric}.  Let $b_i = a_i^{\frac{1-q}{2-s}}$ and consider 
	$w_i (z) = \frac{1}{a_i} u_i (b_i z)$. Note that $\lim_{i \rightarrow \infty}  b_i =0$. For any $R>0$ and $i$ large enough the functions
	$w_i$ are defined in $B(0,R)$ and  $0\leq w_i \leq 1$.
	Note that  from Lemma \ref{U}  it follows that we can assume that $w_i (0)=1$.  It follows from Equation (\ref{R}) that

	\begin{equation}
	b_i^2 a_i^{-1}  \Delta_{g_*} u_i  (b_i z) +b_i^2 a_i^{-1}  K(b_i z) u_i (b_i z) = a_i^{-1+q}  b_i^{2-s} \frac{(a_i^{-1}   u_i(b_i z ))^q}{  \| z \|^s},
	\end{equation}
	
	Since $\Delta_{g_*}  (w_i ) = b_i^2a_i^{-1}  \Delta_h u_i  (b_i z)  +b_i T$ where $T$ is an expression involving first derivatives of $w_i$ we
	can write

	\begin{equation}
	\Delta_{g_*} ( w_i ) +{b_i}  T_1 (w_i ) +  b_i^2 T_0 (w_i )  = \frac{w_i^q}{\| z \|^s},
	\end{equation}
	
	As in the proof of Proposition \ref{regularityC0} we can see that $w_i$ is uniformly bounded in $C^{0,\alpha}$ for some 
	$\alpha \in (0,1)$ and changing $\alpha$ if necessary, that it converges  in $C^{0,\alpha}$
	to some $w\in C^{0,\alpha}$.
	Note that $w(0)=1$.  Moreover as in Proposition \ref{regularityC2loc} we can see that the sequence is uniformly bounded in
	$C^{2,\beta} (K)$ for
	any compact subset  $K$ of $M-S$. We can therefore assume that the sequence converges in $C^{2,\beta}$ on
	compact subsets of $M-S$. Since the metric $g_*$ is close to the Euclidean metric near the origin  and $\lim_{i\to \infty} b_i =0$,  it follows
	from the previous comments  that the limiting function $w\in C^2 (B-0) \cap C^0 (B)$ is a nonnegative solution of

	\begin{equation}
	\Delta w = \frac{w^q}{\| z \|^s}.
	\end{equation}
	
	We know that $w\geq 0$ and $w(0)=1$. It follows that $w>0$. 
	Since $w_i$ are radial functions $w$ is also a radial function. We can let $R\rightarrow \infty$ 
	and therefore $w$ is defined in $\R^{n-k}$. But
	it is known that such solutions do not exist, see \cite[Proposition A]{Subcritical}.
	
	Therefore, this rules out Case 1.
	
	\medskip
	
	{\bf Case 2}: $p\in M-S$. 
	Let $t_i = {\bf {d}} (p_i )$. Assume that $t_i \leq \frac{D}{2}$ for all $i$ (up to taking a subsequence we can assume that this is the case or
	$t_i \geq \frac{D}{2}$ for all $i$, which is treated similarly).
	Let ${\bf{d}} (p)=t = \lim_{i\to\infty} t_i$. It follows that  $t\in (0,D/2 ]$. Consider $\varepsilon >0$ such that 
	$(t-2 \varepsilon , t ] \subset (0,D/2]$. Let $b_i =a_i^{\frac{1-q}{2}}$. Also let $\varphi_i :[0,D]\rightarrow \R$ be the function such that
	$u_i = \varphi_i \circ {\bf d}$. For $i$ large enough  we can 
	consider $\psi_i  :(-\varepsilon b_i^{-1} ,  0]  \rightarrow \R$ given by
	
	$$\psi_i (r)=  \frac{1}{a_i} \varphi_i (  t_i +b_i  r) .$$
	
	\noindent
	Note that $\psi_i (0) =1$ and $\psi_i ' (0)= 0$.
	
	\medskip

	Since $\varphi_i$ satisfies Equation (\ref{ODE}), it follows from a
	direct computation that $\psi_i$ verifies:

	\begin{equation}\label{Normalized}
	\psi_i '' (r) -b_i m(t_i +b_i r ) \psi_i ' (r) - b_i^2 {\bf k}(t_i +b_i r) \psi_i + \frac{\psi_i^q (r)}{(t_i +b_i r)^s} =0 .
	\end{equation}
	
	For any fixed $R>0$ by taking $i$ large enough we can assume that the functions $\psi_i$ are defined on $(-R,0]$. 
	Note that in this interval the functions $r \mapsto m(t_i +b_i r )$, $r \mapsto  {\bf k}(t_i +b_i r)$ and
	$r \mapsto (t_i +b_i r)^{-s}$ are uniformly bounded. Since $\lim_{i \rightarrow \infty}  b_i =0$ and $0\leq \psi_i \leq 1$,
	it follows that the sequence
	$(\psi_i )_i$  is uniformly bounded
	in $C^2 (-R,0]$.  Since $\psi_i$ solves  Equation  (\ref{Normalized}) the sequence is also uniformly bounded in $C^{2,\alpha}  (-R,0]$. 
	Therefore there is a subsequence which converges in $C^2 (-R,0]$  to a solution $\psi$ of the equation 
	
	$$\psi '' (r) +  \frac{\psi^q (r)}{t^s} =0 .$$

	Moreover $\psi \geq 0$,  $\psi (0 )=1$ and $\psi '(0 ) =0$. It is easy to see that such a funcion does
	not exist for $R$ large enough.
	Therefore, Case 2 cannot happen.
	
\end{proof}

\section{Mountain pass solutions}	

In this section, we apply the mountain pass technique to prove Theorem  \ref{mountain pass}.  

Let $f$ be a proper  isoparametric function of $(M,g)$. Let $K\in C^{0,\delta}_f (M)$ such that  the operator $\Delta_g+K$ is coercive. 
Assume that $s\in (0,2)$ and $q>2$. When  $n-k(f)>2$,  we additionally assume that $q<2^*_f(s)$.

Let us consider the functional $J^q : H^2_{1,f}(M)\longrightarrow \R$ defined  by 

\[J^{q}(u):=\frac{1}{2}\int_M|\nabla u|^2_g+ K(x)u^2dv_g-\frac{1}{q}\int_M\frac{|u|^q}{d^s(x)}dv_g
\]
We point out that $H^2_{1,f}(M)$ is a closed subset of $H^2_{1}(M)$ and it is $\Delta$-invariant by the conditions (6) and (7). It follows that
critical points of $J^{q}$ are $f-$invariant  weak solutions to Equation (\ref{q}).

 Note that if $u$ is a critical point of  $J^{q}$,   then $\int_M |\nabla u|^2_g+K(x)u^2dv_g-\int_M\frac{u^{q}}{d^{s}(x)}dv_g=0.$   Therefore,  we have that 
	\[J^{q}(u)=\left( \frac{1}{2}-\frac{1}{q}\right)  \int_M\frac{u^{q}}{d^{s}(x)}dv_g .\]

Then Theorem 1.4 follows from the following proposition.

\begin{Proposition}\label{sequenceofcritical} Let $f$, $K$, $s$ and $q$ be as above. There exists a non-decreasing sequence $\{c_m\}_{m\in \N}$ of critical points of $J^{q}$.
Moreover, $\lim_{m\rightarrow \infty} c_m = \infty$.
\end{Proposition}

The proof of Proposition \ref{sequenceofcritical} relies on the following lemmas.  We denote by $B_{\delta}(0)$ the ball of radius $\delta$ centered at $0$ in $H^2_{1,f}(M)$.

\begin{Lemma}\label{lemma1} Under the assumptions of Proposition \ref{sequenceofcritical}, 
there exist $\delta_0>0$ and $a>0$ such that $J^{q}$ is positive  in $B_{\delta_0}(0)-\{0\}$ and $J^{q}|_{\partial \overline{B_{\delta_0}(0)}}\geq a>0.$

\end{Lemma}

\begin{proof}  From Proposition \ref{R-Kgeneral} it follows that the inclusion map of 
$H^2_{1,f}(M)$ into $L^q(M,d^{-s})$ 
is  continuous. Then, since  $\Delta_g+K$ in coercive,  there exist  constants $A_1>0$  and $A_2>0$  such that  any $u\in H^2_{1,f}(M)$  satisfies
\begin{equation*} \int_M \frac{|u|^q}{d^s(x)}dv_g\leq A_1\|u\|_{H^2_1(M)}^q
\end{equation*}
and 
\begin{equation*}\int_M|\nabla u|^2_g+K(x)u^2dv_g\geq A_2\|u\|^2_{H^2_{1}(M)}.\end{equation*}

Therefore, 
	\begin{align}J^{q}(tu)&=t^2\Big(\frac{1}{2}\int_M|\nabla u|^2_g+K(x)u^2dv_g\Big)-t^p\Big(\frac{1}{q}\int_M\frac{|u|^q}{d^s(x)}dv_g\Big)\nonumber\\
	&\geq t^2\tilde{A}_2\|u\|^2_{H^2_{1}(M)}-t^q\tilde{A_1}\|u\|^q_{H^2_{1}(M)}\nonumber
\end{align}	
	where $\tilde{A}_1=\frac{A_1}{q}$ and $\tilde{A}_2=\frac{A_2}{2}$.

If  $\|u\|_{H^2_{1}(M)}=1$,  it holds that
\begin{equation}\label{ineqfunctenlabola}J^{q}(tu)\geq t^2\tilde{A}_2-t^q\tilde{A}_1.
\end{equation}
Since $q >2$, it follows from Inequality \eqref{ineqfunctenlabola}  that there exists $\delta_0$ such that for any $0<t\leq \delta_0$ $J^{q}(tu)>0$  and 
$J^{q}(\delta_0u) \geq  t_0^2  \tilde{A}_2-t_0^q\tilde{A}_1  >0$. This concludes the proof.

\end{proof}

\begin{Lemma}\label{lemma E} Let $E$ be a finite dimensional subspace of $H^2_{1,f}(M)$. Under the assumptions of Proposition \ref{sequenceofcritical},  
we have that $E\cap \{J^{q}\geq 0\}$ is a  bounded set.
\end{Lemma}

\begin{proof} Let $A_3 >0 $ be such $\int_M|\nabla u|^2_g+K(x)u^2dv_g\leq A_3 \|u\|^2_{H^2_{1}(M)}$ for any $u \in H^2_{1}(M)$.
Since $E$ has finite dimension the set $T=\{ u \in E : \|u\|_{H^2_{1}(M)}=1 \}$ is compact. 
Then, there exists a positive constant $A_4$ such that $\|u\|_{L^q(M,d^{-s})}^q \geq A_4$ for any $u \in T$.

If $u\in T$ we have, 
\[J^{q}(tu) \leq \frac{t^2}{2} A_3-\frac{t^q}{q} A_4.\]
	
It follows that $J^{q}(tu) \leq 0$ if $t^{q-2} \geq \frac{qA_3}{2A_4}$.
	
	\end{proof}

\vspace{1 cm}

Under the assumptions of Proposition \ref{sequenceofcritical}, the functional $J^{q}$ satifies the Palais-Smale condition at any level $r$.  Indeed,  let $\{u_i\}_{i\in\N}$ be a sequence such that 
\begin{equation}\label{PSC}
\lim_{i\to+\infty}J^{q}(u_i)=r\ \mbox{and}\ \lim_{i\to+\infty}(J^{q})'(u_i)=0.
\end{equation}
Using a similar argument to Proposition 14 in  \cite{Mesmar} it can be proved that $\{u_i\}_{i\in \N}$ is  bounded in $H^2_{1,f}(M)$. Hence, up to a subsequence, $\{u_i\}_{i\in \N}$ converges weakly to some $u_r$  in $H^2_{1,f}(M)$. By Proposition \ref{R-Kgeneral}, the inclusion map of $H^2_{1,f}(M)$ into $L^q(M, d^{-s})$  is compact.  Then (up to some subsequence) $\{u_i\}_{i\in \N}$ converges strongly to $u_r$ in $L^q(M, d^{-s})$. Also, $\{u_i\}_{i\in \N}$ converges almost everywhere  in $M$,  therefore $u_r$ if an $f-$invariant function. Finally, we have that 
\[J^q(u_r)=r\ \mbox{and}\ (J^q)'(u_r)=0.\]

\begin{proof}[Proof of Proposition \ref{sequenceofcritical}]
	
Let $\Gamma^*$ be the set of odd homeomorphisms $h: H^2_{1,f}(M) \rightarrow H^2_{1,f}(M)$ such that  $h(0)=0$ and
$h(B_1(0))\subseteq \{J^q\geq 0\}$.

Let $\{e_1, e_2, \dots\}$ be an orthonormal base of $H^2_{1,f}(M)$. We denote with $E_m$ the subspace generated by $\{e_1,\dots, e_m\}$.

 We consider the constants  \[ c_m:=\sup_{h\in \Gamma^*}\min_{u\in E_{m}^{\perp} \cap \partial \overline{B_1(0)}}J^{q}\left(h(u)\right).\]

Since the functional $J^q$ satisfies the Palais-Smale, by Lemma \ref{lemma1}, Lemma \ref{lemma E} ,  the assumptions of 
Theorem 2.13 in \cite{Ambrosetti-R} are satisfied. Then it follows that $c_m$  is a critical value of $ J^{q}$. Then we only need
to prove that $\lim_{m\rightarrow \infty} c_m =\infty$.

Let $N\subset H^2_{1,f}(M)-\{0\}$ be the manifold defined by \[N:=\Big\{  u\in H^2_{1,f}(M)-\{0\}: \frac{1}{2}\int_M |\nabla u|^2_g+K(x)u^2dv_g=\frac{1}{q}\int_M \frac{u^{q}}{d^s(x)}dv_g   \Big\}\]
	
It can be seen that $N$ is homeomorphic to the sphere $\partial{\overline{B_1}(0)}$. For any $u\in H^2_{1,f}(M)-\{0\}$ there exists $t_u>0$ such that $t_uu\in N$. 

We define \[d_m:=\inf_{u\in  E_{m}^{\perp}\cap N}\Big(\int_M |\nabla u|^2_g+K(x)u^2dv_g\Big)^{\frac{1}{2}}.\]

Note that if $u\in  E_{m}^{\perp}-\{0\}$,  then $t_uu\in N\cap E_{m}^{\perp}$. Therefore, 
\begin{equation}\label{dd}d_m\leq \Big(\int_M |\nabla t_uu|^2_g+K(x)(t_uu)^2dv_g\Big)^{\frac{1}{2}}< A t_u \|u\|_{H^2_{1}(M)}\end{equation}
for some $A>0$.

To prove that $\lim_{m\rightarrow \infty} c_m =\infty$ we will prove that $\lim_{m\rightarrow \infty} d_m =\infty$ and that there is a positive constant $\delta$
such that $c_m \geq  \delta d_m^2$.

\begin{itemize}
	\item[\bf Step 1.] We prove that $\lim_{m\rightarrow \infty} d_m =\infty$.  Assume that there is a constant $c$ such that $d_m \leq c$ for all $m$.
Then, there exists a sequence $\{v_m\}\in N\cap E_{m}^{\perp}$ such that \[\int_M |\nabla v_m|^2_g+K(x)v_m^2dv_g\leq c^2+\varepsilon.\]
Since $\Delta_g+K$ is coercive, the sequence $\{v_m\}_{m\in \N}$ must be bounded in $H^2_{1,f}(M)$. Therefore, up to a subsequence, $\{v_m\}_{m\in \N}$ converges weakly to some $v$ in $H^2_{1f}(M)$, and strongly in $L^{q}(M,d^{-s})$ .  However, since $v_m\in E_m^{\perp}$,  it follows that $v$ must be the zero function.	

The inclusion of  $H^2_{1,f}(M)$ into  $L^{q}(M,d^{-s})$ is continuous, then there exists a positive constant $C_0$ such that  
\begin{equation}\label{notbounded}
0<C_0\leq 
\frac{\|v_m\|_{H^2_{1,f}(M)}}{\|v_m\|_{L^{q}(M,d^{-s})}}.
\end{equation}	
Since $v_m\in N$, we have that   $\|v_m\|_{L^{q}(M,d^{-s})}=\Big(\frac{q}{2}\int_M|\nabla v_m|^2_g+K(x)v_m^2dv_g \Big)^{\frac{1}{q}}$.
From the coercivity of $\Delta_g+K$ and \eqref{notbounded}, it follows  that $\|v_m\|_{L^{q}(M,d^{-s})}$ is bounded below by a positive constant, 
which leads to a contradiction.

\item[\bf Step 2.] We prove that there is a positive constant $\delta$
such that $c_m \geq  \delta d_m^2$.
Let  $u\in E_{m}^{\perp}\cap \big(\overline{B_1}-\{0\}\big)$.  
Let $A$ be the positive constant  given in  Inequality \eqref{dd}.

We have,
	 
\begin{align}J^{q} \left( \frac{d_m u}{A} \right)&=
 \frac{d_m^2}{2A^2}\int_M |\nabla u|^2_g+K(x)u^2dv_g-\frac{d_m^{q}}{q A^{q}}
	 \int_M\frac{u^{q}}{d^{s}(x)}dv_g\nonumber\\
	 &=\frac{d_m^2}{2A^2}\int_M|\nabla u|^2_g+K(x)u^2dv_g-\frac{d_m^{q}}{2A^{q}t_u^{q-2}}\int_M\Big( |\nabla u|^2_g+K(x)u^2\Big)dv_g\nonumber\\
	 &=\frac{d_m^2}{2A^2}
	 \left(1-\left(\frac{d_m}{A t_u}\right)^{q-2}\right)\int_M|\nabla u|^2_g+K(x)u^2 dv_g\nonumber
 \end{align}

By Inequality \eqref{dd}, $\left(\frac{d_m}{A t_u}\right)^{q-2}<1$. Then, using the coercivity of $\Delta_g+K$ we obtain that 

\begin{equation}\label{Z}
J^{q} \left( \frac{d_m u}{A} \right) \geq A_1 d_m^2\|u\|^2_{H^2_{1,f}(M)}.
\end{equation}

\vspace{0.4 cm}

For $u\in H^2_{1,f}(M) $ write $u=u_1 + u_2$ with $u_1\in E_{m}^{\perp}$ and $u_2\in E_{m}$.
For $\varepsilon >0$ define  \[h_m(u):=\frac{d_m}{A}u_1+\varepsilon u_2.\]

Then, if $u\in E_{m}^{\perp}\cap \partial\overline{B_1(0)}$ it follows from (\ref{Z}) that $J^q (h_m (u) ) \geq A_1 d_m^2$. If $h_m \in \Gamma^*$ then we would
have that $c_m \geq A_1 d_m^2$. Therefore we only need to prove that $h_m \in \Gamma^*$ for some $\varepsilon >0$. And for this we only need to 
prove that for some $\varepsilon >0$ $h_m (B_1 (0) ) \subset \{ J^q \geq 0 \}$.

 Assume the contrary, that there exists a sequence $\{\varepsilon_i\}$ that goes to zero and 
$u^i = u_1^i + u^i_2 \in B_1 (0) $   such that  
\begin{equation}\label{aa}J^{q} \left(   \frac{d_m}{A} u_1^i+\varepsilon_iu_2^i \right) < 0.
\end{equation}
The sequence $(d_m /A) u_1^i+\varepsilon_i u_2^i$ is bounded in $H^2_{1,f}(M)$, then up to a subsequence, it converges weakly to $u_0$ in $H^2_{1,f}(M)$ and strongly  in $L^{q}(M,d^{-s})$. We have that $J^{q}(u_0)\leq 0$.  Moreover, since $\varepsilon_i \rightarrow 0$ we have that
$u_0 \in E_m^{\perp}$.  By the continuity of the inclusion map of $H^2_{1,f}(M)$ into  $L^{q}(M,d^{-s})$ and the coercivity of $\Delta_g+K$ we have that 
\begin{align}0<L_1&<\frac{\int_M|\nabla( (d_m /A) u_1^i+\varepsilon_i u_2^i)|^2_g+K(x)( (d_m /A)u_1^i+\varepsilon_i u_2^i)^2dv_g}{\| (d_m /A) u_1^i+\varepsilon_i u_2^i\|_{L^{q}(M,d^{-s})}^2}\nonumber\\
&\leq \frac{2}{q} \|(d_m /A)u_1^i+\varepsilon_i u_2^i\|_{L^{q}(M,d^{-s})}^{q-2}.
\nonumber\end{align}
Since $q>2$,  it follows that $u_0\neq 0$.  But then since $u_0\in E_{m}^{\perp}$,  it follows from (\ref{Z}) that $J^{q}(u_0)>0$, which is a contradiction.
\end{itemize}	
	
\end{proof}



\begin{thebibliography}{aa}



\bibitem{Ambrosetti-R} A. Ambrosetti and P. H. Rabinowitz, {\it Dual variational methods in critical point theory and applications}, Journal of Functional Analysis  {\bf 14}  (1973), 349-381.
	
	
\bibitem{Betancourt} A. Betancourt de la Parra,  J. Julio-Batalla and J. Petean, {\it Global bifurcation techniques for Yamabe type equations on Riemannian manifolds}, Nonlinear Analysis  {\bf 202}  (2021), 112-140.



\bibitem{Cheikh} H. Cheikh Ali, {\it The second best constant for the hardy–sobolev inequality on manifolds},
Pacific Journal of Mathematics {\bf 316} (2022), no. 2, 249–276. 




\bibitem{Chen}  W. Chen, {\it Blow-up solutions for Hardy-Sobolev equations on compact Riemannian manifolds}, J. Fixed Point Theory Appl. {\bf 20} (2018), no. 3, Art. 123, 12.


\bibitem{Chi}  Q.-S.,  Chi, {\it The isoparametric story, a heritage of {\'E}lie {Cartan}}, Proceedings of the international consortium of Chinese mathematicians, 2018. Second meeting, Taipei, Taiwan, 2018,   International Press., (2020), 197-260.




\bibitem{Fall} M. M. Fall and E. H. A. Thiam, {\it Hardy-Sobolev Inequality with singularity a curve},
 Topological Methods in Nonlinear
Analysis {\bf 51} (1) (2018), 151-181.

\bibitem{FarkasFaraci} F. Faraci, C. Farkas, {\it A quasilinear elliptic problem involving critical Sobolev exponents}, Collect. math.
{\bf 66} (2015), 243-259.

\bibitem{Farkas} C. Farkas, A. Krist\'{a}ly, A. Mester, {\it Compact Sobolev embeddings on non-compact manifolds via orbit expansions
of isometry groups}, Calc.  Var.  Partial Differentail Equations, {\bf 60} (2021), 60:128.

 \bibitem{Ge-Tang1}  J. Ge and  Z. Tang {\it Isoparametric functions and exotic spheres},  J. Reine Angew. Math. {\bf 683} (2013), 161-180.

\bibitem{Ge-Tang2} J. Ge and  Z. Tang {\it Geometry of isoparametric hypersurfaces in Riemannian manifolds}, Asian J. Math., {\bf 18} (2014), 1:117-125.



\bibitem{Ghoussoub} N. Ghoussoub and X. S.  Kang, {\it Hardy-Sobolev critical elliptic equations with boundary singularities},  AIHP-Anal. Non linéaire {\bf 21} (2004), 767-793.



\bibitem{Ghoussoub-Robert} N. Ghoussoub and F. Robert, {\it The effect of curvature on the best constant in the Hardy-Sobolev inequalities}, GAFA, Geom. funct. anal. {\bf 16} (2006), 1201–1245.


\bibitem{G-Y} N. Ghoussoub and C.  Yuan, {\it  Multiple solutions for quasi-linear {PDEs} involving the critical Sobolev and Hardy exponents}, Trans. Am. Math. Soc. {\bf 352} (2000), 5703--5743.

\bibitem{H} G. Henry, {\it  Isoparametric functions and nodal solutions of the Yamabe equation}, Ann. Glob. Anal. Geom. {\bf 56}, (2019), 2:203-219. 





\bibitem{Jaber} H. Jaber, {\it Hardy-Sobolev equation on compact Riemannian manifolds}, Nonlinear Anal. {\bf 103} (2014),  39-54.

\bibitem{Jaber3} H. Jaber, {\it Optimal Hardy-Sobolev inequalities on compact Riemannian manifolds}, J. Math. Anal. Appl. {\bf 421} (2015), 1869-1888.




\bibitem{Batalla} J.  Julio-Batalla, {\it A note on
sign-changing solutions to supercritical Yamabe-type equations}, Communications on Pure and Applied Analysis {\bf 24} (2025), 2:140-152.





\bibitem{Kang} D. Kang and S. Peng, {\it The existence of positive solutions and sign-changing
solutions for elliptic equations with critical Sobolev-Hardy exponents}, Applied
Mathematics Letters {\bf 17} (2004), 411-416.

\bibitem{Kang-Peng} D. Kang and S. Peng, {\it Existence of solutions for elliptic problems with critical Sobolev–Hardy exponents}, 
Israel J. Math. {\bf 143}
(2004) 281–298.







\bibitem{Mesmar} H.  Mesmar, {\it Best constant and Mountain-Pass solutions for a supercritical Hardy-Sobolev problem in the presence of symmetries}  J. Math. Anal. Appl. {\bf 527 } (2023), Article ID 127437, 30 p.






\bibitem{Subcritical} Q. H. Phan and  P. Souplet, {\it Liouville-type theorems and bounds of solutions of Hardy–H\'{e}non equations},  Journal of Differential Equations
{\bf 252} (2012) , 3:2544-2562.

\bibitem{Qian-Tang} C. Qian and  Z. Tang,{\it  Isoparametric functions on exotic spheres}, Adv. Math. {\bf 272} (2015), 611-629.



\bibitem{Thiam} E. H. A. Thiam, {\ Hardy-Sobolev inequality with higher dimensional singularities}, 
Analysis (Berlin) {\bf 39} (2019), no. 3, 79–96
	
\bibitem{Thiam0} E. H. A. Thiam, {\it 
Hardy and Hardy-Sobolev inequalities on Riemannian manifolds},
IMHOTEP J. Afr. Math. Pures Appl. {\bf 2} (1) (2017),  14-35.


\bibitem{Wang} Q. M. Wang, { \it Isoparametric functions on Riemannian Manifolds. I}, Math. Ann. {\bf{277}}  (1987),  639-646.






\end{thebibliography}
\end{document}